\documentclass[11pt]{amsart}

\usepackage{amsmath,amssymb,latexsym,soul,cite,mathrsfs}
\usepackage{color,enumitem,graphicx}
\usepackage[colorlinks=true,urlcolor=blue,
citecolor=red,linkcolor=blue,linktocpage,pdfpagelabels,
bookmarksnumbered,bookmarksopen]{hyperref}
\usepackage[english]{babel}

\usepackage[left=2.6cm,right=2.6cm,top=2.9cm,bottom=2.9cm]{geometry}
\usepackage[hyperpageref]{backref}

\numberwithin{equation}{section}

\newtheorem{theorem}{Theorem}[section]
\newtheorem{lemma}[theorem]{Lemma}

\newtheorem{proposition}[theorem]{Proposition}

\newtheorem{remark}[theorem]{Remark}

\newtheorem{assertion}[theorem]{Assertion}

\newtheorem{theoremletter}{Theorem}

\newcommand{\R}{{\mathbb R}}
\renewcommand{\S}{{\mathcal S}}

\newcommand{\eps}{\varepsilon}
\renewcommand{\epsilon}{\varepsilon}

\renewcommand{\theta}{{\vartheta}}
\renewcommand{\H}{{\mathcal H}}
\renewcommand{\rightarrow}{\to}

\newcommand{\ud}{\mathrm{d}}
\newcommand{\N}{\mathbb{N}}

\title[Nonlocal problems with Trudinger-Moser nonlinearity]{Ground states of nonlocal scalar field equations  
with \\ Trudinger-Moser critical nonlinearity}  

\author[J.M.\ do \'O]{Jo\~ao Marcos do \'O}
\author[O.H.\ Miyagaki]{Ol\'{i}mpio H.\ Miyagaki}
\author[M. Squassina]{Marco Squassina}

\address[J.M. do \'O]{Department of Mathematics,
Federal University of Para\'{\i}ba
\newline\indent
58051-900, Jo\~ao Pessoa-PB, Brazil}
\email{\href{mailto:jmbo@pq.cnpq.br}{jmbo@pq.cnpq.br}}

\address[O.\ Miyagaki]{Department of Mathematics, 
 Federal University of Juiz de Fora
\newline\indent 
36036-330  Juiz de Fora, Minas Gerais, Brazil}
\email{\href{mailto:olimpio@ufv.br}{olimpio@ufv.br}}

\address[M.\ Squassina]{Dipartimento di Informatica 
Universit\`a degli Studi di Verona,
\newline\indent
C\'a Vignal 2, Strada Le Grazie 15, I-37134 Verona, Italy}
\email{\href{mailto:marco.squassina@univr.it}{marco.squassina@univr.it}}

\thanks{The present research was partially supported by INCTmat/MCT/Brazil.\ Jo\~ao Marcos do \'O was supported by CNPq, CAPES/Brazil, 
Ol\'{i}mpio H.\ Miyagaki was partially supported by CNPq/Brazil and  CAPES/Brazil (Proc 2531/14-3). 
The paper was completed while the second author was visiting the
Dept of Math of Rutgers University, whose hospitality he gratefully
 acknowledges.}
 


\begin{document}

\begin{abstract}
We investigate the existence of ground state solutions for a class of nonlinear scalar field equations defined on whole real line, 
involving a fractional Laplacian and nonlinearities with Trudinger-Moser critical growth.\ We handle the lack of compactness 
of the associated energy functional due to the unboundedness of the domain and the presence of a limiting case embedding. 
\end{abstract}
\maketitle




\section{Introduction and main result}

\noindent
The goal of this paper is to investigate the existence of ground state solutions $u\in H^{1/2}(\R)$ for the 
following class of nonlinear scalar field equations 
\begin{equation}
\label{PS}
(-\Delta)^{1/2} u + u =  f(u) \quad\,\, \text{in $\R$},
\end{equation}
where $f:\R\to\R$ is a smooth nonlinearity in the critical growth range.  Precisely, 
we focus here on the case when $f$ has the {\em maximal growth} 
which allows to study problem \eqref{PS} variationally in the Sobolev space  $u\in H^{1/2}(\R)$, see Section~\ref{prelim-sect}.
We are motivated by the following Trudinger-Moser type inequality due to  Ozawa \cite{Ozawa}.
\begin{theoremletter}
\label{Ozawa}
 There exists $0 < \omega \leq \pi$ such that, for all $\alpha \in (0, \omega)$, there exists $H_{\alpha}>0$ with
\begin{equation}\label{Ia}
\int_{\R} (e^{\alpha u^2}-1) \, \ud x \leq H_{\alpha}\|u\|^{2}_{L^2},
\end{equation}
for all $u\in H^{1/2}(\R)$ with $\|(-\Delta)^{1/4} u\|^2_{L^2}\leq 1$.
\end{theoremletter} 
\noindent
From inequality \eqref{Ia} we have naturally associated notions of {\em subcriticality} and {\em criticality} 
for this class of problems. Precisely, we say that $f:\R \to \R$ has subcritical growth at $\pm\infty$ if 
$$
\limsup_{s\to \pm\infty} \frac{f(s)}{e^{\alpha s^2}-1}=0, \quad \mbox{for all $\alpha >0$},
$$
and has $\alpha_0$-critical growth at $\pm\infty$ if there exists $\omega\in (0,\pi]$ and $\alpha_0\in (0,\omega)$ such that 
\begin{align*}
\limsup_{s\to \pm\infty} \frac{f(s)}{e^{\alpha s^2}-1}&=0, \quad \mbox{for all $\alpha >\alpha_0$}, \\
\limsup_{s\to \pm\infty} \frac{f(s)}{e^{\alpha s^2}-1}&=\pm\infty, \quad\mbox{for all $\alpha <\alpha_0$}.
\end{align*}
For instance let $f$ be given by
\[
f(s)=s^3 e^{\alpha_0 |s|^{\nu}} \quad \mbox{for all $s \in \R$}.
\]
If $\nu <2, \; f$ has subcritical growth, and while if  $\nu =2\;$ and $ \; \alpha_0 \in (0,\omega],  \; f$ has critical growth.
By a {\em ground state} solution to problem \eqref{PS} we mean a nontrivial weak solution of \eqref{PS} with the least possible energy.


\vskip4pt
\noindent
The following assumptions on $f$ will be needed throughout the paper:
\begin{description}
\item[(f1)]  $f:\R\to\R$ is $C^1,$  odd, convex function  on $\R^+,$ and
$$
\lim_{s\to 0}\frac{f(s)}{s}=0.
$$
\item[(f2)] $s\mapsto s^{-1}f(s)$ is an increasing function for $s>0$.
\item[(f3)] there are $q>2$ and $C_q>0$ with
$$
F(s)\geq C_q |s|^q, \quad \text{for all $s\in\R$}.
$$
\item[(AR)]  there exists $\theta>2$ such that 
$$
\theta F(s)\leq sf(s), \quad \text{for all $s\in\R$},\qquad F(s)=\int_0^s f(\sigma)\ud\sigma.
$$
\end{description}

\noindent
The main result of the paper is the following

\begin{theorem}
\label{Theorem}
Let $f(s)$ and $f'(s)s$ have $\alpha_0$-critical growth and satisfy  $\mathbf{(f1)}$-$\mathbf{(f3)}$ and $\mathbf{(AR)}$. Then problem 
\eqref{PS} admits a ground state solution $u\in H^{1/2}(\R)$ provided $C_q$ in $\mathbf{(f3)}$ is large enough. 
\end{theorem}

\noindent
The nonlinearity 
$$
f(s)=\lambda s|s|^{q-2}+|s|^{q-2} s e^{\alpha_0s^2},\quad \text{$q>2$ and $s \in\R$}, 
$$
satisfies all the hypotheses of  Theorem~\ref{Theorem} provided that $\lambda$ is sufficiently large.
More examples of nonlinearities which satisfy the above assumptions can be found in \cite{JMO-exponential}.
In $\mathbb{R}^2$ one can use radial estimates,  then apply, for instance, the Strauss lemma \cite{Strauss}  
to recover some compactness results.  
In $\R$ analogous compactness results fail, but in \cite{Frank}, the authors used  the concentration compactness principle by Lions \cite{willem} for problems with polynomial nonlinearities.
In this paper,  we use the minimization technique over the Nehari manifold in order to get ground state solutions. We adopt some arguments from \cite{AW} combined with those used in \cite{CW,Antonio}. 

\subsubsection{Quick overview of the literature}
In Coti Zelati and Rabinowitz \cite{Coti_Zelati_Rabinowitz} investigated 
\begin{equation}\label{Coti-Rabi}
-\Delta u + V(x) u = f(x,u) \quad \mbox{in $\R^N$}, \quad u\in H^1(\mathbb{R}^N), \quad u>0,
\end{equation}
when $V$ is a strictly positive potential
and $f:\mathbb{R}^N \times \mathbb{R}\rightarrow \mathbb{R}$ is  a periodic function in $x\in \mathbb{R}^N$ and 
$f$ has Sobolev subcritical growth, that is, $f$  behaves at infinity like  
$s^{p}$ with $ 2< p < 2^*-1,$ where $2^*=2N/(N-2)$ is the critical Sobolev exponent, $N\geq 3$.
This was extended or complemented in several ways, see e.g.\  \cite{willem}.
For $ N=2 $ formally $ 2^* \leadsto +\infty,$ but $ H^1(\mathbb{R}^N) \not \hookrightarrow  L^\infty (\mathbb{R}^N).$ Instead, the Trudinger-Moser inequality \cite{M,T} states that $H^1$ is continuously embedded into an Orlicz space defined by the Young function $\phi(t)=e^{\alpha t^2}-1$. In \cite{adimur2,DMR,Def-cpam, LaLu2014}, with the help of Trudinger-Moser embedding, problems in a bounded domain  
were investigated, when the nonlinear term $f$ behaves at infinity like $e^{\alpha s^2}$ for some $\alpha>0$.
We refer the reader to \cite{DOR} for a recently survey on this subject.  
In \cite{CAO1992}  the Trudinger-Moser inequality was extended to the whole $\mathbb{R}^2$ and the authors 
gave some applications to study equations like \eqref{Coti-Rabi} when the nonlinear term has critical growth of Trudinger-Moser type. 
For further results and applications, we would like to mention also \cite{ALDOMI2004,ALSOMO2012,DOMESE2008,RUSA2013} and references therein.
When the potential $V$ is a positive constant and $f(x,s)=f(s)$ for $(x,s) \in \mathbb{R}^N \times \mathbb{R},$  that is the
 autonomous case, the existence of ground states for subcritical nonlinearities was established in  \cite{BELI1983} for $N\geq 3$
 and \cite{BEGAKA1984} for $N=2$ respectively, while in \cite{ALSOMO2012} the critical case for $N\geq 3$ and $N=2$ was treated.
For fractional problem of the form 
\begin{equation}
\label{FP}
(-\Delta)^{s} u + V(x) u =  f(u) \quad\,\, \text{in $\R^N$},
\end{equation}
with $N>2s$ and $ s\in (0,1)$, we refer to \cite{CW,FEQUTA} where positive 
ground states were obtained in subcritical situations. For instance, \cite{CW}  extends the results in 
\cite{BELI1983} to the fractional Laplacian. In \cite{FEQUTA} is obtained regularity and qualitative
 properties of the ground state solution, while in \cite{Simone} a ground state solution is obtained for
 coercive potential. For fractional problems in bounded domains of $\R^N$ with $N>2s$ involving 
critical nonlinearities we cite \cite{barrios,cabretan,Jin,SZY} and  \cite{JMO} for
 the whole space with vanishing potentials.  
 In \cite{Frank} the authors investigated properties of the ground state solutions of 
 $(-\Delta)^{s} u+u=u^p$ in $\R$.
Recently, in \cite{Antonio}, nonlocal problems defined in bounded intervals of the real line involving the square root of the Laplacian and exponential nonlinearities were investigated, using a version of the Trudinger-Moser inequality due to Ozawa \cite{Ozawa}. As it was remarked in \cite{Antonio} the nonlinear problem involving exponential growth with fractional diffusion $(-\Delta)^s$ requires $s=1/2$ and $N=1$. In \cite{JMO-exponential} some nonlocal problems in  $\mathbb{R}$ with vanishing potential, thus providing
compactifying effects, are considered.

\section{Preliminary stuff}
\label{prelim-sect}
\noindent
We recall that
\[
H^{1/2}(\R)=\Big\{u\in L^2(\R): \ \int_{\R^2}\frac{(u(x)-u(y))^2}{|x-y|^2}\ud x\ud y<\infty\Big\},
\]
endowed with the norm
$$
\|u\|=\Big(\|u\|_{L^2}^2+\int_{\R^2}\frac{(u(x)-u(y))^2}{|x-y|^2}\ud x\ud y\Big)^{1/2}.
$$
The square root of the Laplacian, $(-\Delta)^{{1}/{2}},$  of a smooth function $u:\R \to \R$ is defined  by  
$$
{\mathcal F}((-\Delta)^{{1}/{2}}u)(\xi)=|\xi|{\mathcal F}(u)(\xi),
$$ 
where ${\mathcal F}$ denotes the Fourier transform, that is, 
\[
{\mathcal F}(\phi)(\xi)=\frac{1}{\sqrt{2\pi}} \int_{\mathbb{R}} \mathit{e}^{-i \xi \cdot x} \phi (x)  \, \ud x ,  
\]
for functions $\phi$ in  the  Schwartz class.
Also $(-\Delta)^{1/2}u$  can be equivalently represented \cite{nezza} as
$$
(-\Delta)^{1/2} u = -\frac{1}{2\pi}\int_{\R}\frac{u(x+y)+u(x-y)-2 u(x)}{|y|^{2}}\ud y.
$$
Also, in light of \cite[Propostion~3.6]{nezza}, we have
\begin{equation}
\label{equinorm}
\|(-\Delta)^{1/4} u\|^2_{L^2}:=\frac{1}{2\pi}\int_{\R^2}\frac{(u(x)-u(y))^2}{|x-y|^{2}}\ud x \ud y, \quad \text{for all $u\in H^{1/2}(\R)$},
\end{equation}
and, sometimes, we identify these two quantities by omitting the normalization constant $1/2\pi.$
From \cite[(iii) of Theorem 8.5]{LL} we also know that, for any $m\geq 2$,
there exists $C_m>0$ such that
\begin{equation}
\label{emb}
\|u\|_{L^m}\leq C_m\|u\|,
\,\quad\text{for all $u\in H^{1/2}(\R)$}.
\end{equation}

\begin{proposition} \label{finite} The integral 
\begin{equation}\label{New Brunswick}
\int_{\R} (e^{\alpha u^2}-1) \,\ud x
\end{equation}
is finite for any positive $ \alpha $ and $u\in H^{1/2}(\R) $.
\end{proposition}

\begin{proof} 
Let $\alpha_0\in (0,\omega)$ and consider the convex function defined by
\[
\phi(t)=\frac{e^{\alpha_0 t^2}-1}{H_{\alpha_0}}, \quad t\in\R,
\]
where $H_{\alpha_0}>0$ is defined as in Theorem \ref{Ozawa}. We introduce the Orlicz norm induced by $\phi$ by setting
\[
\|u\|_\phi:=\inf\Big\{\gamma>0 \ : \ \int_\R \phi\Big(\frac{u}{\gamma}\Big)\ud x\leq 1\Big\},
\]
and the corresponding Orlicz space $L_{\phi^*}(0,1)$, see the monograph by Krasnosel'ski\u{\i} $\&$ Ruticki\u{\i} 
\cite[Chapter II, in particular p.78-81]{KR} for properties of this space. 
We claim that $\|v\|_\phi\leq\|v\|$, for all $v\in H^{1/2}(\R)$. Let $v\in H^{1/2}(\R)\setminus\{0\}$ and set $w=\|v\|^{-1}v$, so that 
by formula \eqref{equinorm} we conclude
\begin{equation}
	\label{normless1}
\|(-\Delta)^{\frac{1}{4}}w\|_{L^2}=\frac{1}{(2\pi)^{\frac{1}{2}}\|v\|}\Big(\int_{\R^2}\frac{(v(x)-v(y))^2}{|x-y|^{2}}\ud x \ud y\Big)^{1/2}\leq (2\pi)^{-\frac{1}{2}}<1.
\end{equation}
Therefore, in light of Theorem~\ref{Ozawa}, we have
\[
\int_\R \phi\Big(\frac{v}{\|v\|}\Big)\ud x=\int_\R\frac{e^{\alpha_0 w^2}-1}{H_{\alpha_0}} \ud x\leq\|w\|_{L^2}^2\leq 1,
\]
which proves the claim by the very definition of $\|\cdot\|_\phi$. Fix now an arbitrary function $u\in H^{1/2}(\R)$. Hence, there exists a sequence $(\psi_n)$ in $C^\infty_c(\R)$ such that $\psi_n\to  u$ in $H^{1/2}(\R)$, as $n\to\infty$. By the claim this yields $\|\psi_n-u\|_\phi\to 0$, as $n\to\infty$. 
Fix now $n=n_0$ sufficiently large that $\|\psi_{n_0}-u\|_\phi<1/2$. Then we have, in light of \cite[Theorem 9.15, p.79]{KR}, that 
$$
\int_\R \phi(2 u-2\psi_{n_0}) \ud x\leq \|2 u-2\psi_{n_0}\|_\phi<1.
$$
Finally, writing $u=\frac{1}{2} (2 u-2\psi_{n_0})+\frac{1}{2}(2\psi_{n_0})$,  and since 
$$
\int_\R \phi(2\psi_{n_0}) \ud x=\frac{1}{H_{\alpha_0}}\int_\R (e^{4\alpha_0 \psi_{n_0}^2}-1) \ud x=
\frac{1}{H_{\alpha_0}}\int_{{\rm supt}(\psi_{n_0})} (e^{4\alpha_0 \psi_{n_0}^2}-1) \ud x<\infty,
$$
the convexity of $\phi$ yields $\int_\R \phi(u) \ud x<\infty$. Hence, the assertion 
follows by the arbitrariness of $u$. A different proof can be given writing (in the above notations)
$$
\int_{\R} (e^{\alpha u^2}-1) \,\ud x=\int_{\R} (e^{\alpha \psi_n^2}-1) \,\ud x+\int_{\R} (e^{\alpha u^2}-e^{\alpha \psi_n^2}) \,\ud x,
$$
estimating the right-hand side by
$|e^{\alpha u^2}-e^{\alpha \psi_n^2}|\leq 2\alpha(|\psi_n-u|+|\psi_n|)e^{2\alpha |\psi_n-u|^2}e^{2\alpha |\psi_n|^2}|\psi_n-u|,$
using H\"older inequality, the smallness of $\|\psi_n-u\|$ and Theorem~\ref{Ozawa} to conclude, for $n$ large enough.
\end{proof}

\noindent
Define the functional $J: H^{1/2}(\R) \to \R$ associated with problem \eqref{PS}, given by
\[
J(u)=\frac{1}{2}\int_{\R} ( |  (-\Delta)^{1/4} u |^2 + u^2 ) \, \ud x - \int_{\R} F(u) \, \ud x.
\]
Under our assumptions on $f$, by Proposition~\ref{finite}  we can easily see that 
$J$ is well defined. Also, it is standard to prove that $J$ is a $C^1$ functional and 
\[
J'(u)v=\int_{\R}  (-\Delta)^{1/4} u (-\Delta)^{1/4} v \,\ud x+ \int_{\R} uv \,\ud x - \int_{\R} f(u)v \, \ud x , \quad \forall u, v \in H^{1/2}(\R) .
\]
Thus, the critical points of $J$ are precisely the solutions of \eqref{PS}, namely $u \in H^{1/2}(\R)$ with
$$
\int_{\R}  (-\Delta)^{1/4} u (-\Delta)^{1/4} v \,\ud x+ \int_{\R} uv \,\ud x=\int_{\R} f(u)v \, \ud x , \quad \forall v \in H^{1/2}(\R),
$$
is a (weak) solution to \eqref{PS}.

\begin{lemma}
\label{boundmoser}
Let $u\in H^{1/2}(\R)$ and $\rho_0>0$ be such that  $\|u\|\leq \rho_0$. Then 
$$
\int_{\R} (e^{\alpha u^2}-1) \,\ud x\leq \Lambda(\alpha,\rho_0),\quad\text{for every $0<\alpha \rho_0^2<\omega$};
$$
\end{lemma}
\begin{proof}
Let $0<\alpha \rho^2_0<\omega$. Then, by Theorem~\ref{Ozawa}, we have
\begin{equation*}
\int_{\R} \big(e^{\alpha u^2}-1\big) \,\ud x\leq \int_\R \big(e^{\alpha \rho^2_0 \big(\frac{u}{\|u\|}\big)^2}-1\big) \,\ud x 
\leq H_{\alpha \rho^2_0} \frac{\|u\|_{L^2}^2}{\|u\|^2}\leq H_{\alpha \rho^2_0}:=\Lambda(\alpha,\rho_0),
\end{equation*}
since $\|(-\Delta)^{1/4} u\|u\|^{-1}\|^2_{L^2}<1$, see inequality \eqref{normless1}.
\end{proof}

\begin{remark}\rm 
	\label{properts}
From assumptions  $\mathbf{(f1)}$-$\mathbf{(f2)}$ and $ \mathbf{(AR)}$ we see that, for all $s\in\R\setminus\{0\}$,
\begin{align}
s^2f'(s) -sf(s) >  & \,  0,   \label{Mandacaru} \\
f'(s)>&\, 0, &  \label{Manaira} \\
\H(s):=sf(s)-2F(s) > & \,  0,   \label{Tambia}    \\
\mbox{$\H$ is even, and increasing on } &   \R^+ ,  \label{Tambia-1}\\
\H(s) > \H(\lambda s ), \quad \text{for all} &\, \lambda \in (0,1)  \label{Tambia-2}.
\end{align}
\end{remark}

\noindent
Suppose that $u\neq 0$ is a critical point of $J$, that is, $J'(u)=0$, then necessarily $u$ belongs to 
\[
\mathcal{N}:=\left\{ u\in H^{1/2}(\R)\setminus\{0\}: J'(u)u=0 \right\}.
\]
So $\mathcal{N}$ is a natural constraint for the problem of finding nontrivial critical points of $J$. 

\begin{lemma} \label{lemma-1} Under assumptions $\mathbf{(f1)}$-$\mathbf{(f3)}$ and $\mathbf{(AR)}$, $ \mathcal{N} $ satisfies the following properties:
\begin{itemize}
\item[$(a)$] $\mathcal{N}$ is a manifold and $\mathcal{N}\neq \emptyset$.
\item[$(b)$] For $u\in  H^{1/2}(\R)\setminus\{0\}$ with $J'(u)u<0$, there is a unique $\lambda(u)\in (0,1)$ with $\lambda u\in \mathcal{N}$.
\item[$(c)$] There exists $\rho >0$ such that $\|u\|\geq \rho$ for any $u\in \mathcal{N}$.
\item[$(d)$] If $ u \in \mathcal{N} $ is a constrained critical point of $J|_\mathcal{N}$, then $J'(u)=0$ and $u$ solves \eqref{PS}.
\item[$(e)$] $m=\inf_{u\in \mathcal{N}} J(u) >0$.
\end{itemize}
\end{lemma}
\begin{proof} Consider the $C^1$-functional $\Phi : H^{1/2}(\R)\setminus\{0\} \rightarrow \R$ defined by
\[
\Phi(u)=J'(u)u=\|u\|^2-\int_{\R} f(u)u \, \ud x.
\]
Note that $\mathcal{N}=\Phi^{-1}(0)$ and $\Phi'(u)u<0$, if $u\in {\mathcal N}$. Indeed, if $u\in {\mathcal N}$, then
\begin{equation*}
\Phi'(u)u =\int_{\R} \left( f(u)u -  f'(u)u^2 \right)  \ud x <0,
\end{equation*}
where we have used \eqref{Mandacaru}.  Then $ c=0 $ is regular value of $ \Phi $ and consequently $ \mathcal{N}=\Phi^{-1}(0) $ is a $C^1$-manifold,
proving $(a)$. Now we prove $\mathcal{N} \not = \emptyset$ and that $(b)$ holds. Fix $u \in H^{1/2}(\R)\setminus\{0\}$ 
and  consider the function $\Psi:\R^+\rightarrow \R$, 
\[
\Psi(t)=\frac{t^2}{2}\|u\|^2 - \int_{\R} F(tu) \, \ud x.
\]
Then $\Psi'(t)=0$ if and only if $tu\in {\mathcal N}$, in which case it holds
\begin{equation}
\label{akai}
\|u\|^2=\int_{\R} \frac{f(tu)}{t} u \, \ud x.
\end{equation}
In light of \eqref{Mandacaru} the function on the right-hand side of \eqref{akai} is increasing.
Whence, it follows that a critical point of $\Psi$, if it exists, it is {\em unique}.
Now,
there exist $\delta>0$ and $R  >0$ such that 
$$
\Psi(t) >  0 \quad \mbox{if $t \in (0, \delta)$}\quad \mbox{and} \quad
\Psi(t) <  0 \quad \mbox{if $t \in (R, \infty)$}.
$$
In fact, by virtue of $\mathbf{(f3)}$, there exist $C,C'>0$ such that
$$
\Psi(t)=\frac{t^2}{2}\|u\|^2 - \int_{\R} F(tu) \, \ud x\leq Ct^2-C' t^q<0,
$$
provided that $t>0$ is chosen large enough. 
Using $\mathbf{(f1)}$ and the fact that $f$ has $\alpha_0$-Trudinger-Moser critical growth at $+\infty$, 
for some $\alpha\in (\alpha_0,\omega)$ and for any $\eps>0$,  there exists $C_\eps>0$ such that 
$$
F(s) \leq \varepsilon [s^2+s^4(e^{\alpha s^2}-1)]+C_{\varepsilon} s^4, \quad s\in\R.
$$
Then, for any $u\in H^{1/2}(\R)\setminus\{0\}$,
$$
\Psi(t) \geq \frac{t^2}{2}\|u\|^2- \varepsilon t^2\|u\|_{L^2}^2
-C_{\varepsilon} t^4\|u\|_{L^4}^4-\varepsilon t^4\int_{\mathbb{R}}u^4(e^{\alpha (tu)^2}-1)\ud x.
$$
For $0<t<\tau<(\omega/(2\alpha\|u\|^2))^{1/2}$, by Lemma~\ref{boundmoser} and \eqref{emb}, there is $C=C(\|u\|,\alpha)>0$ such that 
$$
\int_{\mathbb{R}}u^4(e^{\alpha(tu)^2}-1)\leq \|u\|_{L^8}^4\Big(\int_{\mathbb{R}}e^{2\alpha \tau^2u^2}-1\Big)^{1/2}\leq C.
$$
Then for some $B,B'>0$, we have 
$$
\Psi(t)\ge Bt^2-B' t^4>0,\quad \mbox{for $t>0$ small enough}.
$$
Thus, we conclude that there exists a unique maximum $t_0=t_0(u)>0$ such that $t_0u\in \mathcal{N}$, and consequently $\mathcal{N}$ is a nonempty set. Given $u\in  H^{1/2}(\R)\setminus\{0\}$ with $J'(u)u<0$, we have 
$$
\Psi'(1)=\|u\|^2 - \int_{\R} f(u) u \, \ud x =J'(u)u  <0,
$$
which implies $t_0<1$. 
Let us prove (c).\
Let $\alpha\in (\alpha_0,\omega)$ and $\rho_0>0$ with
$\alpha \rho^2_0<\omega$. By the growth conditions on $f$, there exists
$r>1$ so close to $1$ that $r\alpha\rho_0^2 <\omega$, $\ell>2$ and $C>0$ with
$$
f(s)s\leq \frac{1}{4}s^2+C(e^{r\alpha s^2}-1)^{1/r}|s|^\ell,\quad \text{for all $s\in\R$}.
$$
Let now $u\in {\mathcal N}$ with $\|u\|\leq\rho\leq\rho_0$. 
Then, by Lemma~\ref{boundmoser} and \eqref{emb}, we have for  $u\in {\mathcal N}$
\begin{align}
\label{controllo-picc}
0&=\Phi(u) \geq \|u\|^2-\frac{1}{4}\|u\|_{L^2}^2-C\int_{\R}\big(e^{r\alpha u^2}-1\big)^{1/r}|u|^\ell \,\ud x  \notag\\
& \geq \frac{3}{4}\|u\|^2-C\Big(\int_{\R}\big(e^{r\alpha u^2}-1\big)\,\ud x\Big)^{1/r}  \Big(\int_\R |u|^{r'\ell} \,\ud x)^{1/r'} \\
& \geq \frac{3}{4}\|u\|^2-C\|u\|^\ell, \notag
\end{align}
which yields
$$
0<\hat\rho:=\big(\frac{3}{4C}\big)^{1/(\ell-2)}\leq \|u\|\leq \rho,
$$
a contradiction if $\rho<\min\{\hat\rho,\rho_0\}$.\ Then $u\in {\mathcal N}$ implies $\|u\|\geq\min\{\hat\rho,\rho_0\}$, proving (c).\
Concerning (d), if $u\in {\mathcal N}$ is a minimizer, then $J'(u)=\lambda\Phi'(u)$
for some $\lambda\in\R$. Testing with $u$ and recalling the previous conclusions yields $\lambda=0,$ hence the assertion.
Finally, assertion (e) follows by condition $\mathbf{(AR)}$ and (c), since $u\in {\mathcal N}$ implies $J(u)\geq (1/2-1/\vartheta)\|u\|^2
\geq (1/2-1/\vartheta)\rho^2>0$.
\end{proof}

\begin{lemma}\label{Cabo Branco} Let $(u_n)\subset  \mathcal{N}$  be a minimizing sequence for $J$ on $\mathcal{N}$, that is, 
\begin{equation}\label{Rosa}
J'(u_n)u_n=0 \quad \mbox{and}\quad J(u_n)\rightarrow m:=\inf_{u\in \mathcal{N}} J(u) \quad \mbox{as $n\to\infty$},
\end{equation}
then the following facts hold
\begin{itemize}
\item[$(a)$]  $(u_n)$ is bounded in $H^{1/2}(\R)$. Thus, up to a subsequence, 
$u_n\rightharpoonup u$ weakly in $H^{1/2}(\R)$. 
\item[$(b)$] $\limsup_{n} \|u_n\| < \rho_0$, for some $\rho_0>0$ sufficiently small.
\item[$(c)$] $(u_n)$ does not converge strongly to zero in $L^\sigma(\R)$, for some $\sigma>2$.
\end{itemize}
\end{lemma}
\begin{proof} Let $(u_n)\subset H^{1/2}(\R)$  satisfying \eqref{Rosa}. 
Using $\mathbf{(AR)}$ condition, we have for $\vartheta>2$,
\begin{equation}
\label{est-m}
m+\mathit{o}(1)=  J(u_n)\geq  \, \frac{\|u_n\|^2}{2}-\frac{1}{\theta}\int_{\mathbb{R}}  f(u_n)u_n  \, \ud x
=  \Big(\frac{1}{2}-\frac{1}{\theta}\Big)\|u_n\|^2,
\end{equation}
which implies (a).
To prove (b) we use assumption $\mathbf{(f3)}$
and the fact that, by \eqref{emb},
\begin{equation}
\label{Caja}
\S_q := \inf_{v\in H^{1/2}(\R)\setminus\{0\}}\S_q(v)>0,\qquad \S_q(v)=\frac{\|v\|}{\|v\|_{L^q}}.
\end{equation}
Let $(u_n)\subset  \mathcal{N}$ and $u\in \mathcal{N}$ satisfying \eqref{Rosa}. Then inequality
\eqref{est-m} yields
\begin{equation}\label{Araca}
\limsup_n \|u_n\|^2 \leq \frac{2\theta}{\theta-2}m.
\end{equation}
Notice that, for every $v\in H^{1/2}(\R)\setminus\{0\}$, arguing as for the proof of (b) of Lemma~\ref{lemma-1},
one finds $t_0>0$ such that $t_0 v \in \mathcal{N}$. Hence
\[
m  \leq J(t_0v) \leq \max_{t\geq 0} J(t v).
\]
Now, using assumption $\mathbf{(f3)}$ and formula \eqref{Caja}, for every 
$\psi\in H^{1/2}(\R)\setminus\{0\}$, we can estimate 
\[
\begin{aligned}
m \leq & \max_{t\geq 0} J(t \psi) 
\leq  \max_{t\geq 0} \left( \frac{t^2}{2} \|\psi\|^2 - C_q t^q \|\psi\|^q_{L^q} \right)\\
\leq & \max_{t\geq 0} \left( \frac{\S_q(\psi)^2}{2} t^2 \|\psi\|^2_{L^q} - C_q t^q \|\psi\|^q_{L^q} \right)\\
= & \left( \frac{1}{2} -\frac{1}{q} \right) \frac{\S_q(\psi)^{2q/(q-2)}}{(qC_q)^{2/(q-2)}},
\end{aligned}
\]
which together with \eqref{Araca} implies that 
\begin{equation*}
\limsup_n \|u_n\|^2 \leq \frac{2\theta}{\theta-2} \left( \frac{1}{2} -\frac{1}{q} \right) \frac{\S_q(\psi)^{2q/(q-2)}}{(qC_q)^{2/(q-2)}}.
\end{equation*}
Taking the infimum over $\psi\in H^{1/2}(\R)\setminus\{0\}$, we get
\begin{equation}\label{Umbu}
\limsup_n \|u_n\|^2 \leq \frac{\theta}{\theta-2} \frac{q-2}{q}  \frac{\S_q^{2q/(q-2)}}{(qC_q)^{2/(q-2)}}<\rho_0^2,
\end{equation}
provided $C_q$ is large enough, proving (b).  Let us prove (c). By Lemma~\ref{lemma-1} (part (c)) we have
\[
\|u_n\|^2=\int_{\mathbb{R}}f(u_n)u_n \, \ud x \geq \rho^2>0.
\]
In view of assertion (b) the norm $\|u_n\|$ is small (precisely, we can
 assumed that $r\alpha\|u_n\|^2< r\alpha \rho_0^2<\omega$ for $r$ very close to $1$).
Arguing as in the proof of \eqref{controllo-picc},  we can find $\eps\in (0,1)$ and $C>0$ such that 
\begin{align*}
\|u_n\|^2=\int_{\mathbb{R}}f(u_n)u_n \, \ud x &\leq \eps\|u_n\|_{L^2}^2+C\int_{\R}\big(e^{r\alpha u_n^2}-1\big)^{1/r}|u_n|^\ell \,\ud x  \notag\\
& \leq \eps\|u_n\|^2+C\Big(\int_{\R}\big(e^{r\alpha u_n^2}-1\big)\,\ud x\Big)^{1/r}  \|u_n\|_{L^{r'\ell}}^\ell  \\
&\leq \varepsilon \|u_n\|^2 + C\|u_n\|_{L^{r'\ell}}^\ell,
\end{align*}
which implies 
\[
0<(1-\varepsilon)\rho^2\leq (1-\varepsilon)\|u_n\|^2 \leq  C\|u_n\|_{L^{r'\ell}}^\ell,
\]
and, consequently, $(u_n)$ cannot vanish in $L^{r'\ell}(\R)$, as $n\to\infty$.
This concludes the proof.
\end{proof}


\noindent
Next, we formulate a Brezis-Lieb type lemma in our framework.

\begin{lemma}
\label{BL}  
Let $(u_n)\subset  H^{1/2}(\R)$  be a sequence such that $u_n\rightharpoonup u$ weakly in $H^{1/2}(\R)$
and $\|u_n\|<\rho_0$ with $\rho_0>0$ small.
Then,  as $n\to\infty$, we have 
\begin{align*}
\int_{\mathbb{R}}f(u_n)u_n  \, \ud x &= \int_{\mathbb{R}}f(u_n-u)(u_n-u)  \, \ud x+\int_{\mathbb{R}}f(u)u  \, \ud x+o(1),  \\
\int_{\mathbb{R}} F(u_n)\, \ud x &=\int_{\mathbb{R}} F(u_n-u)\, \ud x+\int_{\mathbb{R}} F(u)  \, \ud x+o(1).
\end{align*}
\end{lemma}
\begin{proof}
We shall apply \cite[Lemma 3 and Theorem 2]{BL}.\  Since $f$ is convex on $\R^+$ and by the properties 
collected in Remark \ref{properts},  we have that the functions
$F(s)$ and $ G(s):=f(s)s$ are convex on $\R$ with $F(0)=G(0)=0.$
We let $\alpha\in(\alpha_0,\omega)$ and  $\rho_0\in (0,1/2)$ with $\alpha \rho^{2}_0< \omega$. Thus, by 
Lemma \ref{boundmoser}, we have
\begin{equation}\label{limite}
\sup_{n\in\N} \int_{\R} \big( e^{\alpha u^{2}_{n}}-1 \big) \ud x < \infty.
 \end{equation}
Choose $k\in (1,\frac{1-\rho_0}{\rho_0})$ and let $\eps>0$ with $\varepsilon<1/k$. Then, in light of \cite[Lemma 3]{BL}, the functions
$$
\phi_{\epsilon}(s):=j(ks)-kj(s)\geq 0
\,\,\,\quad 
\psi_{\epsilon}(s):=|j( C_{\epsilon}s)|+|j(- C_{\epsilon}s)|,\quad s\in\R, \quad C_{\varepsilon}=\frac{1}{\varepsilon(k-1)},
$$
satisfy the inequality
\begin{equation}\label{(3)}
|j(a+b)-j(a)|\leq \epsilon \phi_{\epsilon}(a) + \psi_{\epsilon}(b), \quad \forall a, b \in \R, 
\end{equation}
and, if $v_n:=u_n -u$ and $u_n$ satisfying \eqref{limite}, we claim that 
\begin{itemize}
\item [(i)] $v_n\rightarrow 0$ almost everywhere;
\item [(ii)]$j(u) \in L^1(\R);$ 
\item [(iii)] $\int_{\R} \phi_{\epsilon}(v_n) dx \leq C$ for some constant $C$ independent of $n\geq 1$;
\item [(iv)] $\int_{\R} \psi_{\epsilon}(u) dx < \infty,$ \,\, for all $\epsilon>0$ small.
\end{itemize}
Assuming this claim, then, by \cite[Theorem 2]{BL}, it holds 
\begin{equation}
\label{bn}
\lim_n\int_{\R}|j(u_n)-j(v_n)-j(u)| dx=0,
\end{equation}
with $j=F$ and with $j=G.$ 
Next we are going to prove the claim. Item (i) follows by the week convergence of $(u_n)$. 
To prove (ii) it is enough to use Proposition~\ref{finite} (see the growth conditions below). To check (iii) for $j=F$ and $j=G$, 
we find $\alpha\in(\alpha_0,\omega)$, $D>0$ and $q>2$ such that 
\begin{align}
F(s) \leq \;&  \big(s^2 + e^{\alpha  s^2}-1 \big) +D|s|^{q}, \quad\text{for all  $s\in \R$},  \label{cresceF} \\
G(s) \leq \; &  \big(s^2 + e^{\alpha  s^2}-1 \big) +D|s|^{q}, \quad \text{for all  $s\in \R$}, \label{cresceG} \\
|f(s)| \leq \; &  \big(s + e^{\alpha  s^2}-1 \big) +D|s|^{q-1}, \quad \text{for all  $s\in \R$},  \label{cresceFa} \\
|f'(s)s| \leq \; &  \big(s + e^{\alpha  s^2}-1 \big) +D|s|^{q-1}, \quad \text{for all  $s\in \R$}.  \label{cresceFb}
\end{align}
We claim that $\phi_{\epsilon}(v_n)$ verifies (iii). First let us consider the case $j=F$, that is, 
$$
\phi_{\epsilon}(v_n)=F(kv_n)-kF(v_n).
$$ 
In fact, by the Mean Value Theorem, there exists $\theta \in (0,1)$ with $w_n=v_n(k(1-\theta)+\theta)$ 
such that
\begin{align*} 
\phi_{\epsilon}(v_n)&=F(kv_n)-F(v_n)+F(v_n)-kF(v_n)\\
&= f(w_n)v_n (k-1) + (1-k)F(v_n)\leq f(w_n)v_n (k-1),
\end{align*}
since $k>1$ and $F\geq 0$. Analogously, for $j=G$, we have 
\begin{align*} 
\phi_{\epsilon}(v_n)&=G(kv_n)-G(v_n)+G(v_n)-kG(v_n)\\
&=f'(w_n)w_n v_n (k-1) + f(w_n)v_n(k-1)+(1-k)f(v_n)v_n \\
&\leq f'(w_n)w_n v_n (k-1) + f(w_n)v_n(k-1),
\end{align*}
since $k>1$ and $f(s)s\geq 0$ for all $s\in\R$.
Thus, to prove (iii) for $F$ and $G$, it is sufficient to see that
\begin{equation}
\label{2tt}
\sup_{n\in\N}\int_{\R}f(w_n)v_n  \ud x <\infty, 
\quad \sup_{n\in\N}\int_{\R}f'(w_n)w_n v_n  \ud x <\infty.
\end{equation}
We know that
$$
\|u_n\|^2=\|v_n\|^2+ \|u\|^2 + o(1), \,\, \text{as $n\to\infty$},
$$
so that $\limsup_n \|v_n\|\leq \rho_0\ .$ In turn, by the choice of $k$, we also have 
$$
\limsup_n \| w_n \|=\|v_n\| (k(1-\theta)+\theta)\leq \rho_0  (k(1-\theta)+\theta)\leq \rho_0  (k+1)<1.
$$
Since $\alpha_0<\alpha<\omega$, we can find $m>1$ very close to $1$ such that $m\alpha<\omega$. Then, by \eqref{cresceFa}, we get
\begin{align*}
\int_{\R}f(w_n)v_n &\leq\int_{\R}|w_n||v_n|\ud x + \int_{\R} (e^{\alpha  w_n^2}-1)|v_n| \ud x +D\int_{\R} |w_n|^{q-1} |v_n| \ud x  \\
&\leq \|w_n\|_{L^2}\|v_n\|_{L^2}+D\|w_n\|_{L^q}^{q-1}\|v_n\|_{L^q}+\Big(\int_{\R} (e^{m\alpha  w_n^2}-1)\ud x\Big)^{1/m}\|v_n\|_{L^{m'}}  \\
&\leq C\|w_n\|\|v_n\|+C\|w_n\|^{q-1}\|v_n\|+C\Big(\int_{\R} (e^{m\alpha  w_n^2}-1)\ud x\Big)^{1/m}\|v_n\|   \\
 &\leq C+C\Big(\int_{\R} (e^{m\alpha  w_n^2}-1)\ud x\Big)^{1/m}\leq C.  
\end{align*}
The last integral is bounded via Lemma~\ref{boundmoser}, since $\|w_n\|\leq 1$ and $m\alpha<\omega$.
The second term in \eqref{2tt} can be treated in a similar fashion, using 
the growth condition \eqref{cresceFb} in place of \eqref{cresceFa}.
%
We claim that $\psi_{\epsilon}$ verifies (iv) for both $F$ and $G$.
It suffices to prove
$$
\int_{\R}F( C_{\epsilon} u)\ud x < \infty, \quad\text{for all $\epsilon >0$}.
$$
By \eqref{cresceF} this occurs since by Proposition \ref{finite}, we have
$$
\int_{\R} (e^{\alpha C^{2}_{\epsilon}u^2} -1) \ud x <\infty.
$$
Analogous proof holds for $G$ via \eqref{cresceG}. 
We can finally apply \cite[Theorem 2]{BL} yielding \eqref{bn}. Thus
$$
\int_{\R}j(u_n) \ud x=\int_{\R}j(v_n) \ud x+\int_{\R}j(u) \ud x+ o(1), 
$$
for $j=F$ and $j=G.$ This concludes the proof.
\end{proof}

\noindent
The previous Lemma~\ref{BL} yields the following useful technical results.

\begin{lemma}\label{Cabo Branco-miramar} Let $ (u_n)\subset H^{1/2}(\R) $ be as in Lemma~\ref{Cabo Branco}
then for $v_n=u_n-u$ we have
$$
J'(u)u+\liminf_n J'(v_n)v_n =0,
$$ 
so that either $J'(u)u\leq 0$ or $\liminf_n J'(v_n)v_n <0$.
\end{lemma}
\begin{proof}
Recalling that $v_n=u_n-u$, we get $\|u_n\|^2=\|v_n\|^2 + \|u\|^2+o(1).$ Then by Lemma \ref{BL},
$$
\int_{\mathbb{R}}f(u_n)u_n  \, \ud x =\int_{\mathbb{R}}f(v_n)v_n  \, \ud x+\int_{\mathbb{R}}f(u)u  \, \ud x +\mathit{o}(1).
$$
Since $u_n \in \mathcal{N}$, by using the above equality, the assertion follows.
%
\end{proof}

\begin{lemma}\label{lemma-3} Let $(u_n)\subset  \mathcal{N}$  be a minimizing sequence for $J$ on $\mathcal{N}$,  such that  $u_n\rightharpoonup u$ weakly in $H^{1/2}(\R)$  as $n\rightarrow \infty$.  If $u\in  \mathcal{N}$, then $J(u)=m$.
\end{lemma}
\begin{proof} Let $(u_n)\subset  \mathcal{N}$ and $u\in \mathcal{N}$ as above, thus 
\[
m+\mathit{o}(1)=J(u_n)-\frac{1}{2} J'(u_n)u_n=\frac{1}{2}\int_{\R} \H(u_n) \, \ud x
\]
which together with Fatou's lemma (recall that \eqref{Tambia} holds) implies 
\begin{equation*}
m =  \frac{1}{2}\liminf_{n} \int_{\R}\H(u_n) \ud x
\geq  \frac{1}{2}\liminf_{n} \int_{\R}\H(u) \ud x= J(u)-\frac{1}{2} J'(u)u=J(u),
\end{equation*}
which yields the conclusion.
\end{proof}

\section{Proof of Theorem~\ref{Theorem} concluded}

\noindent
Let $(u_n)\subset  \mathcal{N}$  be a minimizing sequence for $J$ on $\mathcal{N}$.  
From Lemma~\ref{Cabo Branco} (a),  $(u_n)$  is bounded in  $H^{1/2}(\R)$. Thus, up to a subsequence, we have $u_n\rightharpoonup u$ weakly in $H^{1/2}(\R)$.
\begin{assertion} There exist a sequence $(y_n) \subset \R$ and constants $\gamma, \, R>0$ such that 
\[
\liminf_{n}\int_{y_n-R}^{y_n+R} |u_n|^2 \, \ud x \geq \gamma >0
\]
\end{assertion}
\noindent If not, for any $R>0$,
\[
\liminf_{n} \sup_{y\in \R} \int_{y-R}^{y+R} |u_n|^2 \, \ud x = 0.
\]
Using a standard concentration-compactness principle due to P.L.\ Lions (it is easy to see that the argument remains valid for the case studied here) 
we can conclude that $u_n \rightarrow 0$  in $L^q(\R)$ for {\em any $q>2$}, which is a contradiction with Lemma~\ref{Cabo Branco} (c).
\bigskip
\noindent
Define $\bar u_n (x)=u_n(x+y_n)$. Then $J(u_n)=J(\bar u_n)$ and  without of loss generality we can assume $y_n=0$ for any $n$. Notice that $(\bar u_n)$ is also a minimizing sequence for $J$ on $\mathcal{N}$, which it is bounded and satisfies  
\[
\liminf_{n}\int_{-R}^{R} |\bar u_n|^2 \, \ud x \geq \gamma, \quad \mbox{for some} \quad \gamma >0,
\]
and $\bar u_n\rightharpoonup \bar u$ weakly in $H^{1/2}(\R)$, then $\bar u \not = 0 \; ( u \not =0)$.

\begin{assertion} \label{Formosa}
 $J'(u)u=0$. 
\end{assertion} 

\noindent
If Assertion~\ref{Formosa} holds, then combining (d) of Lemma~\ref{lemma-1} 
and Lemma~\ref{lemma-3} we have the result.

\smallskip
\noindent
We shall now prove Assertion~\ref{Formosa}. Suppose by contradiction that $J'(u)u\not =0$. 
\vskip4pt
\noindent
$\bullet$ If $J'(u)u<0$, by Lemma~\ref{lemma-1}\,(b), there exists $ 0 < \lambda <1 $ such that $\lambda u \in \mathcal{N}$. Thus
\begin{equation}\label{sete}
\lambda \|u\|^2=\int_{\R} f( \lambda u) u \, \ud x.
\end{equation}
Using \eqref{Tambia} in combination with Fatou's Lemma, we obtain
\begin{equation*}
m =\liminf_n   \frac{1}{2}\int_{\R} \H(u_n) \, \ud x 
\geq \frac{1}{2}\int_{\R} \H(u) \, \ud x>  \frac{1}{2}\int_{\R} \H( \lambda u) \, \ud x 
=  J(\lambda u)-\frac{1}{2}J'(\lambda u)\lambda u
=  J(\lambda u),
\end{equation*}
which implies that $J(\lambda u) <m$ and, hence, a contradiction. Here we have used \eqref{Tambia-2}. 
\vskip4pt
\noindent
$\bullet$ If $J'(u)u>0$, by Lemma~\ref{Cabo Branco-miramar}, we get $\liminf_n J'(v_n)v_n <0$. Taking a
subsequence, we have $J'(v_n)v_n <0$, for $n$ large. By Lemma~\ref{lemma-1}\,(b), 
there exists $\lambda_n \in (0,1)$ such that $\lambda_n v_n \in \mathcal{N}$.

\begin{assertion} 
$\limsup_n \lambda_n <1$.
\end{assertion}
\vskip2pt
\noindent
If $\limsup_n \lambda_n =1$, up to a sub-sequence, we can assume that $\lambda_n \rightarrow 1$, then 
\begin{equation*}
J'(v_n) v_n =  J'(\lambda_n v_n) \lambda_n v_n +o(1).
\end{equation*}
This follows provided that
\begin{equation}
\label{toprove1}
\int_\R f(v_n) v_n \ud x=\int_\R f(\lambda_n v_n) \lambda_n v_n \ud x+o(1).
\end{equation}
In fact, notice that if $\eta_n:=v_n+\tau v_n (\lambda_n-1)$ for some $\tau\in (0,1)$, it follows
$$
f(v_n) v_n- f(\lambda_n v_n) \lambda_n v_n=\big(f'(\eta_n)\eta_n+f(\eta_n)\big)v_n (1-\lambda_n).
$$
Since $\|\eta_n\|=\|v_n+\tau v_n (\lambda_n-1)\|\leq \lambda_n\|v_n\|\leq\rho_0$, it follows by
arguing as for the justification of formula \eqref{2tt}, that
$$
\sup_{n\in\N}\int_{\R}|f'(\eta_n)\eta_n+f(\eta_n)||v_n| \ud x<\infty,
$$
so that \eqref{toprove1} follows, since $\lambda_n\to 1$.
Since $ \lambda_n v_n \in \mathcal{N}$ we have $J'(\lambda_n v_n) \lambda_n v_n=0$ which implies that 
\[
J'(v_n) v_n =  \mathit{o}(1),
\]
which is a contradiction with $\lim_n J'(v_n)v_n <0$. Thus, up to subsequence, we may assume that  $\lambda_n \rightarrow \lambda_0 \in(0,1).$
Arguing as before, from \eqref{Tambia-2} we infer
\begin{equation*}
m+o(1)=\frac{1}{2}\int_{\mathbb{R}} \H(u_n) \, \ud x\geq  \frac{1}{2}\int_{\mathbb{R}} \H(\lambda_n u_n) \, \ud x,
\end{equation*}
since $\H(u_n) \geq \H(\lambda_n u_n).$ 
By means of Lemma~\ref{BL} applied to $w_n=\lambda_n u_n$ (whose norm is small, being 
smaller than the norm of $u_n$) and $w=\lambda_0 u$ 
we have in turn
$$
\int_{\R} \H(\lambda_n u_n) \, \ud x=\int_{\R} \H(\lambda_n u_n-\lambda_0 u) \, \ud x+\int_{\R} \H(\lambda_0 u) \, \ud x+o(1).
$$
Furthermore, we have
\begin{equation}
\label{toprove2}
\int_{\R} \H(\lambda_n u_n-\lambda_0 u) \, \ud x=\int_{\R} \H(\lambda_n v_n) \, \ud x+o(1).
\end{equation}
In fact, notice that $\lambda_n u_n-\lambda_0 u=\lambda_n v_n+\gamma_n u$, 
where $\gamma_n:=\lambda_n-\lambda_0\to 0$ as $n\to\infty$. We have
$$
\H(\lambda_n u_n-\lambda_0 u)-\H(\lambda_n v_n)=\H'(\hat \eta_n)u\gamma_n,  \,\,\quad
\hat\eta_n:=\tau u\gamma_n+\lambda_n v_n
$$
for $\tau\in (0,1)$ and $\|\hat\eta_n\|=\|\tau u\gamma_n+\lambda_n v_n\|\leq \gamma_n \|u\|+\lambda_n\|v_n\|\leq \rho_0$ for $n$ large. 
Then, arguing as for the justification of \eqref{2tt}, we get
$$
\sup_{n\in\N}\int_{\R} |\H'(\hat \eta_n)||u| \ud x\leq
\sup_{n\in\N}\int_{\R}|f'(\hat\eta_n)\hat\eta_n+f(\hat\eta_n)||v_n| \ud x<\infty,
$$
which yields \eqref{toprove2} since $\gamma_n\to 0$ as $n\to\infty$.
Therefore, we obtain
\begin{align*}
m+\mathit{o}(1)&\geq  \frac{1}{2}\int_{\R} \H(\lambda_n v_n)\, \ud x  
+  \frac{1}{2}\int_{\R}  \H(\lambda_0 u)\, \ud x  \\
&=J(\lambda_n v_n)-\frac{1}{2} J'(\lambda_n v_n)\lambda v_n+\frac{1}{2}\int_{\R}  \H(\lambda_0 u)\, \ud x
=J(\lambda_n v_n)+\frac{1}{2}\int_{\R}  \H(\lambda_0 u)\, \ud x.
\end{align*}
Since $u \neq 0,$ we have $\int_{\R}  \H(\lambda_0 u)\, \ud x>0$. Then
$J(\lambda_n v_n) < m$ for large $n$, a contradiction.  \qed





\bigskip
\bigskip


\bigskip

\end{document}